\documentclass[11pt]{amsart}

\usepackage{amsmath}
\usepackage{amssymb}
\usepackage{pdfsync}
\usepackage[colorlinks,citecolor=red,pagebackref,hypertexnames=false]{hyperref}
\usepackage{tikz}

\newcommand*\circled[1]{\tikz[baseline=(char.base)]{
            \node[shape=circle,draw,inner sep=2pt] (char) {#1};}}

\newcommand{\R}{{\mathbb R}}       % Field of real numbers
\newcommand{\II}{{\mathcal I}}

\newcommand{\WW}{{\mathcal W}}
\newcommand{\PP}{{\mathcal P}}

\newcommand{\com}{{\Omega}}       % Field of complex numbers
\newcommand{\om}{{\omega}}

 \def\cH{\mathcal{H}}
\def\d{\partial}

\newcommand{\diam}{\mathop{\rm diam}}
\newcommand{\dist}{{\rm dist}}

\newcommand{\esssup}{\operatorname{ess\,sup}}

\newcommand{\rf}[1]{{(\ref{#1})}}

\newcommand{\supp}{\operatorname{supp}}

\newcommand{\ve}{{\varepsilon}}
\newcommand{\vv}{{\vspace{2mm}}}

\newcommand{\wt}[1]{{\widetilde{#1}}}

\newcommand{\noi}{\noindent}
\newcommand{\lip}{{\rm Lip}}

\newtheorem{theorem}{Theorem}[section]
\newtheorem{lemma}[theorem]{Lemma}

\newtheorem{claim}{Claim}

\newtheorem*{lemma*}{Lemma}

\theoremstyle{definition}
\newtheorem{definition}[theorem]{Definition}

\theoremstyle{remark}
\newtheorem{rem}[theorem]{\bf Remark}

\numberwithin{equation}{section}

\newcommand{\brem}{\begin{rem}}
\newcommand{\erem}{\end{rem}}

\begin{document}

\title[Singular sets for harmonic measure]{Singular sets for harmonic measure on locally flat domains with locally finite surface measure}

\author{Jonas Azzam}
\address{Departament de Matem\`atiques\\ Universitat Aut\`onoma de Barcelona \\ Edifici C Facultat de Ci\`encies\\
08193 Bellaterra (Barcelona, Catalonia) }
\email{jazzam@mat.uab.cat}
\author{Mihalis Mourgoglou}
\address{Departament de Matem\`atiques\\ Universitat Aut\`onoma de Barcelona and Centre de Reserca Matem\` atica\\ Edifici C Facultat de Ci\`encies\\
08193 Bellaterra (Barcelona, Catalonia) }
\email{mmourgoglou@crm.cat}
\author{Xavier Tolsa}
\address{ICREA and Departament de Matem\`atiques\\ Universitat Aut\`onoma de Barcelona \\ Edifici C Facultat de Ci\`encies\\
08193 Bellaterra (Barcelona, Catalonia) }
\email{xtolsa@mat.uab.cat}
%\thanks{Supported in part by the grants RTG DMS 08-38212  and DMS-0856687}
\thanks{The three authors were supported by the ERC grant 320501 of the European Research Council (FP7/2007-2013). X.T. was also supported by 2014-SGR-75 (Catalonia), MTM2013-44304-P (Spain), and Marie
Curie ITN MAnET (FP7-607647).}

\maketitle

\begin{abstract}
 A theorem of David and Jerison asserts that harmonic measure is absolutely continuous with respect to surface measure in
 NTA domains with Ahlfors regular boundaries. 
 We prove that this fails in high dimensions if we relax the Ahlfors regularity assumption by showing that, for each $d>1$, there exists a Reifenberg flat domain $\Omega\subset \mathbb{R}^{d+1}$ with $\cH^{d}(\d\Omega)<\infty$ and a subset $E\subset \d\Omega$ with positive harmonic measure yet zero $\cH^{d}$-measure. In particular, this implies that a classical theorem 
  of F. and M. Riesz fails in higher dimensions for this type of domains.
\end{abstract}

\tableofcontents

\section{Introduction}

The F. and M. Riesz theorem states that, for simply connected planar domains whose boundary has finite length, harmonic measure and arc-length are mutually absolutely continuous. The obvious generalization to higher dimensions is false due to examples of Wu and Ziemer: they construct topological two-spheres in $\mathbb{R}^{3}$ with boundaries of finite Hausdorff measure $\cH^{2}$ where either harmonic measure is not absolutely continuous with respect to $\cH^{2}$ \cite{Wu86} or $\cH^{2}$ is not absolutely continuous with respect to harmonic measure \cite{Z}, respectively. In spite of this, there has still been interest in narrowing down some sufficient conditions for when the F. and M. Riesz theorem still holds in higher dimensions. In fact, the study of the relationship between harmonic measure and the geometric and metric properties of domains is a very active area of research. See for example \cite{AMNHT}, \cite{BL}, \cite{HUT} or \cite{KPT}.

Recall that a {\it nontangentially accessible domain} (or {\it NTA domain}) $\Omega\subset \mathbb{R}^{d+1}$ is a connected open set for which the following hold:
\begin{enumerate}
\item $\Omega$ is a {\it $C$-uniform domain}, meaning for all $x,y\in \Omega$ there is $\gamma\subset \Omega$ for which $\cH^{1}(\gamma)\leq C|x-y|$ and $\dist(z,\Omega^{c})\geq \dist(z,\{x,y\})/C$ for all $z\in \gamma$, and
\item $\Omega$ satisfies the {\it $C$-exterior corkscrew condition}, meaning for all $\xi\in \d\Omega$ and $r>0$ there is $B(z,r/C)\subset B(\xi,r)\backslash \Omega$.
\end{enumerate}

In \cite{DJ90}, David and Jerison show that if $\Omega\subset \mathbb{R}^{d+1}$ is an NTA domain and $\d\Omega$ is {\it Ahlfors regular}, meaning there is $A>0$ so that
\[r^{d}/A\leq \cH^{d}(B(\xi,r)\cap \d\Omega)\leq Ar^{d}\mbox{ for all }\xi\in \d\Omega \mbox{ and } r\in (0,\diam\Omega),\]
then not only do we have $\om \ll \cH^{d}|_{\d\Omega} \ll \om$, but they are in fact $A_{\infty}$-equivalent. At first look, Ahlfors regularity seems superfluous for establishing absolute continuity, and in some sense it is: in \cite{Badger12}, Badger shows that if one merely assumes $\cH^{d}|_{\d\Omega}$ is locally finite and $\Omega\subset \mathbb{R}^{d+1}$ is NTA, then we still have $\cH^{d}|_{\d\Omega} \ll \om$. He also shows that  $\om \ll \cH^{d}|_{\d\Omega}\ll \om$ on the set
\[\{\xi\in \d\Omega: \liminf_{r\rightarrow 0}\cH^{d}(B(\xi,r)\cap \d\Omega)/r^{d}<\infty\}\]
but asks whether or not mutual absolute continuity holds on the whole boundary of $\Omega$ and not just on the above set (see Conjecture 1.3 in \cite{Badger12}). However, this turns out not to be true in general: we show that there exist domains with locally finite perimeter where $\om$ is not absolutely continuous with respect to $\cH^{d}|_{\d\Omega}$ even if we assume stronger conditions than the NTA property, such as Reifenberg flatness. 

By a domain in $\R^{d+1}$ we mean an open connected set. Given $A,B\subset \R^{d+1}$, we denote by $\dist_H(A,B)$ the Hausdorff distance between $A$ and $B$.

\begin{definition}[Reifenberg flat domain]\label{defreif}
Let $\Omega\subset \R^{d+1}$ be an open set, and let $0<\delta<1/2$, $r_0>0$. We say that $\Omega$
is a $(\delta,r_0)$-Reifenberg flat domain if it satisfies the following conditions:
\begin{itemize}
\item[(a)] For every $x\in\partial \Omega$ and every $0<r\leq r_0$ there exists a hyperplane $\PP(x,r)$ containing 
$x$ such that
$$\dist_H\bigl(\partial \Omega\cap B(x,r),\,\PP(x,r)\cap B(x,r)\bigr) \leq \delta\,r.$$

\item[(b)] For every $x\in\partial \Omega$, one of the connected components of 
$$B(x,r_0)\cap \bigl\{x\in\R^{d+1}:\dist(x,\PP(x,r_0))\geq 2\delta\,r_0\bigr\}$$
is contained in $\Omega$ and the other is contained in $\R^{d+1}\setminus\Omega$.
\end{itemize}

If $\Omega$ is $(\delta,r_0)$-Reifenberg flat for every $r_0>0$, we say that it is 
$(\delta,\infty)$-Reifenberg flat.
\end{definition}

Note that the topological condition (b) is asked only for the scale $r=r_0$. However, from the definition one can check that
the same comparability condition also holds for $r\leq r_0$ (see \cite[Proposition 2.2]{Kenig-Toro} or 
\cite[Lemma 5]{LMS}, for example).
Further, we remark that the condition (b) is implied by (a) if one assumes that both $\Omega$ and $\partial
\Omega$ are connected, as shown by David \cite{David-approx}.

We can now state the main result.

\begin{theorem}\label{maintheorem}
For all $d\geq 2$, $\delta>0$ small enough and $r_{0}>0$, there is a $(\delta,r_{0})$-Reifenberg flat domain $\Omega\subset \mathbb{R}^{d+1}$ and a set $E\subset \d\Omega$ such that $\cH^{d}|_{\d\Omega}$ is a Radon measure and if $\om$ is the harmonic measure for $\Omega$ with respect to a fixed pole in $\Omega$, then 
\[\om(E)>0=\cH^{d}(E).\]
\end{theorem}

We give a sketch of the proof: We rely on the existence of Wolff snowflakes from either \cite{Wolff} or \cite{LVV05}, which are NTA domains $\Omega\subset \mathbb{R}^{d+1}$ for which 
\begin{equation}
 \lim_{r\rightarrow 0} \frac{\log \om(B(\xi,r))}{\log r }<d \mbox{ for } \om \mbox{ a.e. }\xi\in \d\Omega.
 \end{equation}
where $\om$ denotes harmonic measure for $\Omega$ with respect to a fixed pole $z_{0}\in \Omega$. By some measure theory, this means we can find a compact set $E\subset \d\Omega$ with $\om(E)>0$ and constants $\alpha,r_{0}>0$ so that 
\begin{equation}\label{e:d-e}
\om(B(\xi,r))>r^{d-\alpha}\mbox{ for all }r\in (0,r_{0}).
\end{equation} 

We then build a Reifenberg flat domain $\Omega^{+}\supset \Omega$ so that $\d\Omega^{+}\supset E$ and use \eqref{e:d-e} to control the $\cH^{d}$-measure of $\d\Omega^{+}$. Moreover, if $\om_{\Omega^{+}}$ is harmonic measure for $\Omega^{+}$ with respect to the same pole $z_{0}$, then by the maximum principle, we have
\[\om_{\Omega^{+}}(E)\geq \om(E)>0=\cH^{d}(E).\]

The lemma for constructing this domain is not particular to our problem and may be of independent interest, see Section 2.

\vspace{2mm}

As usual, in this paper  we will use the letters $c,C$ to denote
absolute constants which may change their values at different
occurrences. Constants with subscripts, such as $c_1$, do not change their values
at different occurrences.
The notation $A\lesssim B$ means that
there is some fixed constant $c$ such that $A\leq c\,B$. So $A\sim B$ is equivalent to $A\lesssim B\lesssim A$. 
If we want to write explicitly 
the dependence on some constants $c_1$ of the relationship such as ``$\lesssim$'', we will write
 $A\lesssim_{c_1} B$. We will assume all these implicit constants in these inequalities depend on $d$ and frequently omit the subscript.\\

We give many thanks to the anonymous referees for pointing out several errors and helping the authors improve the readability of this manuscript.

%\section{Preliminaries}

\section{The enlarged domain $\Omega^+_\ve$}

We will assume $\Omega$ is $(\delta,r_0)$-Reifenberg flat for some $\delta<1/2$ sufficiently small. From now on, for $x\in\partial\Omega$ and $0<r\leq r_0$, we will denote by $N_{x,r}$ a unit vector, normal to $\PP(x,r)$, with its sign chosen so that $x+\frac34r N_{x,r}\not\in \Omega$ and $x-\frac34r N_{x,r}\in \Omega$. That $N_{x,r}$ can be taken in this way is guaranteed by the property (b) in the definition above, which holds for all $0<r\leq r_0$. In fact, from this one can deduce that
\begin{equation}\label{eqinc1}
B(x+\tfrac34r N_{x,r},\,\tfrac r{10})\subset \Omega^c \qquad \mbox{and}\qquad
B(x-\tfrac34r N_{x,r},\,\tfrac r{10})\subset \Omega.
\end{equation}

Let us mention that, by Theorem 3.1 of \cite{Kenig-Toro}, there is $\delta_0=\delta_0(d)$ such that
if $\Omega$ is $(\delta,r_0)$-Reifenberg flat, with $0<\delta\leq \delta_0$, then both $\Omega$ and $\R^{d+1} \setminus \overline{\Omega}$ are uniform domains.

\begin{definition}[Whitney-type cubes]
For an open set $\Omega\subsetneq \R^{d+1}$ that is $(\delta,r_{0})$-Reifenberg flat and $K\geq4$, we denote by $\WW_{K}(\Omega)$ the set of maximal dyadic cubes $Q\subset \Omega$ such that $\diam KQ\leq r_{0}$ and $K Q\cap \Omega^{c}=\varnothing$. These cubes have disjoint interiors and can be easily shown to satisfy the following properties:
\begin{enumerate}
\item[(a)] $\min\{r_{0},\dist(x,\Omega^{c})\}/K\lesssim \ell(Q)\lesssim \min\{r_{0},\dist(x,\Omega^{c})\}/K$ for all $x\in Q$, where $\ell(Q)$ denotes the side length of the cube.
%\item $(\frac{K-1}{2}-\sqrt{d}\frac{\lambda-1}{2})\ell(Q)\leq \dist(x,\Omega^{c})\leq  (1+K+(\lambda-1)/2)\diam Q$ for all $x\in \lambda Q$ if $\lambda\geq 1$ is close enough to $1$ (depending on $d$ and $K$)
\item[(b)]  If $Q,R\in \WW_{K}(\Omega)$ and $\frac K4 Q\cap \frac K4 R\neq\varnothing$, then $\ell(Q)\sim_{K,d}\ell(R)$.
\item[(c)] $\sum_{Q\in \WW_{k}(\Omega)}\chi_{\frac K4 Q}\lesssim_{K,d}\chi_{\Omega}$.

%\item $\frac{K-1}{2}\ell(Q)\leq \dist(x,\Omega^{c})\leq (1+K)\,\diam Q$ for all $x\in Q$,
%\item $(\frac{K-1}{2}-\sqrt{d}\frac{\lambda-1}{2})\ell(Q)\leq \dist(x,\Omega^{c})\leq  (1+K+(\lambda-1)/2)\diam Q$ for all $x\in \lambda Q$ if $\lambda\geq 1$ is close enough to $1$ (depending on $d$ and $K$)
%\item If $Q,R\in \WW_{K}(\Omega)$ and $Q\sim R$ (i.e. they are adjacent), then $\ell(Q)\sim_{K,d}\ell(R)$.
%\item $\sum_{Q\in \WW_{k}(\Omega}\chi_{2\lambda Q}\lesssim_{K,d}\chi_{\Omega}$ for $\lambda> 1$ sufficiently small (depending on $K$ and $d$).

\end{enumerate}
\label{d:Whitney}
\end{definition}

These are similar to the usual Whitney cubes, but we restrict their size.

% ***************************************************************************

% ***************************************************************************

%\section{The enlarged domain $\Omega^+_\ve$}

%Let $\Omega\subsetneq \R^{d+1}$ be a domain and let $E\subsetneq \partial \Omega$ be a closed set. 
%For any given $\tau>0$, we say that $E$ is $\tau$-dense in $\d\Omega$ if  $\dist(x,E)\leq \tau$ for all $x\in\partial\Omega$.

Let $0<\ve<1/100$
be some small constant and $E\subsetneq \d\Omega$ be any closed set. Denote by $\II $ the family of cubes $Q\in\WW_{\ve^{-2}}(E^c)$ such that
$Q\cap\partial\Omega\neq\varnothing$. Notice that
$$\ell(Q)\lesssim \ve^2\,\dist(Q,E)\qquad\mbox{for all $Q\in \II $}$$
and
$$\partial \Omega\setminus E \subset \bigcup_{Q\in \II } Q.$$
For each $Q\in \II $, fix some point $z_Q\in Q\cap\partial \Omega$ and set
$$B_Q = B(z_Q,\ve\,\min\{r_{0},\dist(z_Q,E)\}).$$
Notice  that $\ell(Q)\sim \ve\,r(B_Q)$.
Then we consider the domain
$$\Omega^+_\ve = \Omega \cup \bigcup_{Q\in \II } B_Q.$$
Our main objective in this section consists in proving the following.

\begin{lemma}\label{lemreif1}
Let $r_0\in(0,\infty]$ and let $\ve>0$ be small enough. There exists $\delta_{0}=\delta_{0}(\ve)>0$ such that if $\Omega\subset \R^{d+1}$ is $(\delta,r_0)$-Reifenberg flat for some $\delta\in(0,\delta_{0})$, $E\subsetneq\partial \Omega$ 
is closed and 
$\Omega^+_\ve$ is as above,  then $E\subset \partial \Omega^+_\ve$ and
$\Omega^+_\ve$ is $(c\ve^{1/2},r_0/2)$-Reifenberg flat.
\end{lemma}

First we will prove the following auxiliary result.

\begin{lemma}\label{lemreif2}
With the same notation and assumptions as in Lemma \ref{lemreif1}, for every $Q\in \II $, there exists a function
$f_Q:\PP(z_Q,30r(B_Q))\cap 10B_Q\to \PP(z_Q,30r(B_Q))^\bot$ such that, assuming after a suitable rotation
that $L_Q:= \PP(z_Q,30r(B_Q))=\R^{d}\times\{0\}$, the following holds:
\begin{itemize}
\item[(a)] $\Omega_\ve^+\cap 10B_Q=\{x\in 10B_Q:x_{d+1}<f_Q(\tilde x)\}$, where
$\wt x=(x_1,\ldots,x_d)$.

\item [(b)] For all $\wt x\in \PP(z_Q,30r(B_Q))\cap 10B_Q$, $|f_Q(\wt x) - r(B_Q)|\leq c\,\ve\,
r(B_Q)$.

\item[(c)] The function $f_Q$ is Lipschitz with Lipschitz constant at most $c\,\ve^{1/2}$.

\item[(d)] For all $x\in 10B_{Q}\cap L_{Q}$, $(x,f_Q( x))\in\partial \Omega^+_\ve$
\end{itemize}
\end{lemma}

\begin{proof} We will assume that $\delta\ll\ve$.
To simplify notation we write $$r_Q=r(B_Q)=\ve\,\min\{r_{0},\dist(z_Q,E)\}.$$
%If $\ve$ is small enough, then $r_Q\leq r_0/100$, for all $Q\in \II$. Indeed, in the case $r_0<\infty$, since $E$ is $r_0$-dense in $\Omega$ and $z_Q\in Q\cap \partial\Omega$, 
%we have 
%$$r_Q=\ve\,\dist(z_Q,E) \leq \ve\,r_0\leq \frac{r_0}{100},
%$$
%aasuming $\ve\leq 1/100$.
Further, we suppose that the component $\{x\in 30B_Q:x_{d+1}<-60\delta\,r_Q\}$ is contained in $\Omega$ (recall that the property (b) in Definition
\ref{defreif} holds for all $r\leq r_0$).

\begin{claim}\label{claim1}
There is a finite subfamily $\II _Q\subset \II $ such that 
\begin{equation}\label{eq1}
20 B_Q\cap \partial\Omega_\ve^+ \subset \bigcup_{P\in \II _Q}\partial B_P.
\end{equation}
Further, for every $P\in \II _Q$, we have
\begin{equation}\label{eq2}
|z_P-z_Q|\leq 30\,r_Q,
\end{equation}
\begin{equation}\label{eq3}
\dist(z_P,L_Q)\leq 30\delta\,r_Q,
\end{equation}
and
\begin{equation}\label{eq4}
|r_P-r_Q|\leq c_1	\,\ve\,r_Q.
\end{equation}
\end{claim}

Indeed, $\partial \Omega_\ve^+\subset E \cup \bigcup_{P\in \II }\partial B_P$ holds by definition, which obviously implies
\begin{equation}\label{eqf32}
\partial \Omega_\ve^+\setminus E\subset \bigcup_{P\in \II }\partial B_P.
\end{equation}
We denote by $ \II _Q$ the subfamily of the cubes $P\in \II $ such that $20 B_Q\cap{ B_P}\neq \varnothing$,
so that
\begin{equation}\label{eqf322}
20 B_Q\cap \partial\Omega_\ve^+ \subset E\cup \bigcup_{P\in \II _Q}\partial B_P.\end{equation}
By definition, $E\cap 20B_Q=\varnothing$, from which \eqref{eq1} readily follows from \eqref{eqf32} and \eqref{eqf322}.

Suppose that $P\in \II _Q$.
Also, by the definition of $r_P$ and $r_Q$ along with $20 B_Q \cap B_P$, and since $\min\{r_{0},\dist(z,E)\}$ is $1$-Lipschitz, we have
$$|r_P-r_Q| \leq \ve\,|z_P-z_Q|\leq \ve(r_P + 20r_Q).$$
One can check that this ensures that
\begin{equation}\label{eqf33}
|r_P-r_Q| \leq c\,\ve\,r_Q,
\end{equation}
and thus \rf{eq4} holds.
This implies that $r_P\sim r_Q$, and thus $\ell(P)\sim\ve r_Q$.
From this condition, taking into account that the cubes $P\in \II_Q$ are pairwise disjoint and all of
them intersect $20B_Q$ (by the definition of $\II_{Q}$), it follows that $\II_Q$ is finite.

Since for $\ve$ small enough we have $r_P\leq 2r_Q$, we deduce that
$$|z_P-z_Q|\leq r_P + 20r_Q\leq 22r_Q,$$
which yields \rf{eq2}.
On the other hand, \rf{eq3} follows from the Reifenberg flatness of $\Omega$
and the fact that $z_P\in\partial\Omega\cap 30B_Q$.

\begin{claim}\label{claim2}
If $x\in20 B_Q\cap\partial\Omega_\ve^+$, then
\begin{equation}\label{eqdja9}
|x_{d+1}-r_Q|\leq c_2\ve\,r_Q,
\end{equation}
for some absolute constant $c_2>0$.
Moreover, every $x\in20 B_Q\cap\d\Omega_\ve^+$ satisfies
\begin{equation}\label{eqdja9'}
(1-c_2\ve)\,r_Q\leq x_{d+1}\leq (1+c_2\ve)\,r_Q.
\end{equation}
\end{claim}

Let $x\in20 B_Q\cap\partial\Omega_\ve^+$.
Notice first that, by \rf{eq1}, \rf{eq3} and \rf{eq4},
\begin{equation}\label{eqdja10}
|x_{d+1}| = \dist(x,L_Q)\leq \max_{P\in \II _Q} \bigl[\dist(z_P,L_Q) +r_P\bigr]
\leq 30\delta \,r_Q + (1+c_1\ve)r_Q \leq (1+2c_1\ve)r_Q,
\end{equation}
since we are assuming $\delta\ll\ve$.

Now we will show that 
\begin{equation}\label{eqdja11}
|x_{d+1}|=\dist(x,L_Q)\geq (1-c_2\ve)\,r_Q.
\end{equation}
To this end, notice that, by the Reifenberg flatness of $\Omega$, if $\wt x$ is the projection of $x$ onto $L_{Q}$, there exists $x'\in \partial\Omega$ such that $|x'-\wt x|\leq 30\delta\,r_Q$. Let $P\in \II _Q$ be the Whitney cube such that $x'\in P$. Then,
$$|x-z_P|\leq |x-\wt x| + |\wt x-x'| + |x'-z_P| \leq |x_{d+1}| + 30\delta\,r_Q + c\,\ell(P).$$
Recalling that 
$\ell(P)\sim\ell(Q)\sim\ve\,r_Q$ and using that $|x-z_P|\geq r_P$ (because $x$ is not in the
interior of $B_P$), we deduce that
\begin{align*}
|x_{d+1}|&\geq |x-z_P| - 30\delta\,r_Q - c\,\ell(P) \geq r_P -30\delta\,r_Q - c\,\ve\,r_Q\geq
(1-c_1\ve-30\delta- c\,\ve)\,r_Q,
\end{align*}
which proves \rf{eqdja11} since we assume that $\delta\ll\ve$.

From \rf{eqdja11}, we infer that 
$$x\in 30B_Q\setminus U_{60\delta r_Q}(L_Q),$$
using again that $\delta\ll\ve \ll 1$, where $U_\epsilon(F)$ stands for the $\epsilon$-neighborhood of the set $F$.
Since $x\in\Omega^c$ (as $\partial\Omega^+_\ve\subset \Omega^c$), we infer that $x_{d+1}>0$ (because
we are assuming that $\{x\in 30B_Q:x_{d+1}<-60\delta\,r_Q\}$ is contained in $\Omega$ by Definition \ref{defreif} and \eqref{eqinc1}).
Hence $|x_{d+1}|=x_{d+1}$ and then \rf{eqdja10} and \rf{eqdja11} yield \rf{eqdja9}.

The second statement in the claim follows easily from \eqref{eqdja9}.

\begin{claim}\label{claim3}
For $\wt x\in 10 B_Q\cap L_Q$, let
\begin{equation}\label{eqdja20}
f_Q(\wt x) = \max\bigl\{t\in \R:(\wt x,t)\in20 \overline {B_Q}\cap\overline{\Omega^+_\ve}\bigr\}.
\end{equation}
The function $f_Q:10 B_Q\cap L_Q\to\R$ is well defined, $(\wt x,f_Q(\wt x))\in\partial \Omega^+_\ve$, and
\begin{equation}\label{eqdja21}
|f_Q(\wt x) - r_Q|\leq c_2\ve\,r_Q.
\end{equation}
\end{claim}

To see this, let $\wt x\in 10 B_Q\cap L_Q$ and consider the points
$x_1= (\wt x,-2r_Q)$ and $x_2= (\wt x,2r_Q)$. Then we have 
$$x_1\in 20B_Q\cap\Omega\subset 20B_Q\cap\Omega^+_\ve
\quad\mbox{ and }\quad x_2\in 20B_Q\setminus \Omega^+_\ve.$$  
The first statement follows from the fact that $\{x\in 30B_Q:x_{d+1}<-60\delta\,r_Q\}$ is contained in $\Omega$ and the second one from \rf{eqdja9'}.
Hence, there exists some $t_0\in [-2r_Q,2r_Q]$ such that $(\wt x,t_0)\in\partial \Omega^+_\ve$, and thus
the maximum in \rf{eqdja20} is taken over a non-empty set. Further, \rf{eqdja9'} also tells us that
$(\wt x,t)\not\in\Omega^+_\ve$ for every $t\geq2r_Q$, and thus it follows that
$(\wt x,f_Q(\wt x))\in\partial\Omega^+_\ve$.

The estimate \rf{eqdja21} is an immediate consequence of \rf{eqdja9}.

\begin{claim}\label{claim4}
Let $x=(\wt x,t)\in 10 B_Q$ with $t<f_Q(\wt x)$. Then $x\in \Omega_\ve^+$.
\end{claim}

By Claim \ref{claim1} and Claim \ref{claim3}, there exists some $P\in \II _Q$ such that
$(\wt x,f_Q(\wt x))\in \partial B_P$. Let $\wt L_Q$ be a hyperplane parallel to 
$L_Q$ passing through $z_P$.
Let $y\in\R^{d+1}$ be the reflection of $(\wt x,f_Q(\wt x))$ with respect to $\wt L_Q$. It is clear that $y\in\partial B_P$ and thus the (open) segment with end points $(\wt x,f_Q(\wt x))$ and $y$ is contained in $B_P$ and thus in $\Omega^+_\ve$. That is to say,
$$(\wt x,t)\in \Omega^+_\ve \quad\mbox{if $t\in(y_{d+1},f_Q(\wt x))$.}$$
By symmetry, $y_{d+1}+f_Q(\wt x)  \leq 2\dist(z_P,L_Q)$ and since $\dist(z_P,L_Q)\leq 30\delta\,r_Q\leq \ve\,r_Q$, using also \rf{eqdja21} we infer that 
$$y_{d+1}\leq -f_Q(\wt x) + 2\ve\,r_Q \leq - (1-c\,\ve)r_Q\leq -\frac12\,r_Q.$$ 
On the other, since $(\wt x,t)\in \Omega \subset \Omega^+_\ve$ for  $(\wt x,t)\in 10B_Q$ with $t<-r_Q/2$,
the claim follows.

\begin{claim}\label{claim5}
The function $f_Q$ is Lipschitz with Lipschitz constant not exceeding $c\,\ve^{1/2}$.
\end{claim}

From Claim \ref{claim3} and Claim \ref{claim1} we deduce that
\begin{equation}\label{eqdja30}
f_Q(\wt x) = \max\{t:\exists \;P\in \II _Q \mbox{ such that }(\wt x,t)\in \partial B_P\}.
\end{equation}
Observe that if $x=f(\tilde{x})\in \partial B_P \cap \partial\Omega^+_\ve$, then
$$f_Q(\wt x) = \sqrt{r_P^2-|\wt x - \wt z_P|^2} + z_{P,d},$$
where we wrote $z_P = (\wt z_P, z_{P,d})$. Further, since $f_Q(\wt x)\geq (1-c_2\ve)r_Q$ by \eqref{eqdja21}, clearly we can write $f_Q(\tilde{x}) = g_P(\tilde{x})$, where $g_P:L_Q\to L_Q^\bot$ is defined by
$$g_P(\wt x) = \max\Bigl(\sqrt{r_P^2-|\wt x - \wt z_P|^2} + z_{P,d}\,,\,(1-c_2\ve)r_Q\Bigr)
\quad \mbox{
if $|\wt x- \wt z_P|\leq r_P$,}$$
 and $g_P(\wt x) = (1-c_2\ve)r_Q$ otherwise.
 From \rf{eqdja30} we infer that, for all $\wt x\in L_Q\cap 10B_Q$,
$$f_Q(\wt x) = \max_{P\in \II _Q} g_P(\wt x).$$
So to prove the claim it suffices to show that each function $g_P$ is Lipschitz with Lipschitz constant
at most $c\,\ve^{1/2}$. To check that this holds, notice that $g_P$ is continuous and differentiable 
a.e., with
\begin{equation}\label{eqdja31}
|\nabla g_P(\wt x)| = \frac{|\wt x- \wt z_P|}{\sqrt{r_P^2-|\wt x - \wt z_P|^2}}\quad \mbox{
if \;$\sqrt{r_P^2-|\wt x - \wt z_P|^2} + z_{P,d}>(1-c_2\ve)\,r_Q$,}
\end{equation}
and $|\nabla g_P(\wt x)|=0$ a.e.\ otherwise. The condition on the right hand side  of \rf{eqdja31}
implies that
$$\sqrt{r_P^2-|\wt x - \wt z_P|^2} >(1-c_2\ve)\,r_Q - z_{P,d} \geq (1-c_3\ve)r_P,$$
by \rf{eq4} and \rf{eq3}. This gives
$$|\wt x - \wt z_P|^2 < r_P^2- (1-c_3\ve)^2\,r_P^2 \leq 2c_3\ve\,r_P^2.$$
Plugging this estimate into \rf{eqdja31} we get
$$|\nabla g_P(\wt x)| \leq c\,\ve^{1/2},$$
which implies that $\lip(g_P)\leq c\,\ve^{1/2}$, as wished.
This concludes the proof of the claim.

\vv
The lemma follows from the statements in the claims above.
\end{proof}

\vspace{2mm}

\begin{proof}[\bf Proof of Lemma \ref{lemreif1}]
First we show that $E\subset \partial \Omega^+_\ve$. Notice that 
$E\subset \overline \Omega \subset \overline {\Omega^+_\ve}$. So it suffices to show that for each $x\in E$ there exists
a sequence of points $\{x_k\}_k\subset (\Omega_\ve^+)^c$ such that $x_k\to x$ as $k\to \infty$. To construct this
sequence, for each $0<r\leq r_0$ consider the ball $B(x,r)$, so that
by the Reifenberg flatness of $\Omega$, $\partial \Omega\cap B(x,r)\subset U_{\delta r}(\PP(x,r))$.
Further, any ball $B_Q$, with $Q\in \II $, which intersects $B(x,r)$ satisfies
$$r_Q \leq \ve\,\dist(z_Q,E)\leq \ve\,|z_Q-x|\leq \ve\,(\sqrt{d+1}r_Q+r),$$
 where as in the previous lemma, we write $r_Q=r(B_Q)\leq\ve\,\dist(z_Q,E)$. Then it follows that $r_Q< 2\ve\,r$ for $\ve$ small. Hence we deduce that
$B(x,r)\cap \overline{\Omega_\ve^+}\subset B(x,r)\cap U_{2\ve r}(\overline\Omega)$.
From \rf{eqinc1}, we deduce that if $\ve$ is small enough, then
$$y_r:= x+ \frac34r\,N_{x,r}\in(\Omega_\ve^+)^c.$$
Thus, setting $r=1/k$ and $x_k=y_{1/k}$, we are done.

\vv
Now we have to show that 
$\Omega^+_\ve$ is $(c\ve^{1/2},r_0/2)$-Reifenberg flat. 
By construction, if $x\in \partial\Omega^+_\ve$, then either $x\in E$ or there exists some ball
$B_Q$ such that $x\in\partial B_Q$, and so $\dist(x,\partial\Omega)\leq r_Q$.

To show that the properties (a) and (b) in the Definition \ref{defreif} hold for $\Omega_\ve^+$ and the ball $B(x,r)$, with $0<r<r_0/2$ 
and with $c\ve^{1/2}$ instead of $\delta$, we distinguish several cases:

\vv
\noi{\bf Case 1.} Suppose that $r\geq \ve^{-1/2}r_Q$ for every $Q\in \II $ such that 
\begin{equation}\label{bqbx2r}
B_Q\cap B(x,2r)\neq\varnothing.
\end{equation}
From the discussion  in the previous paragraphs, it turns out that there exists $x'\in\partial\Omega$ such that
\begin{equation}\label{eq101}
|x-x'|\leq \sup\{r_Q: Q\in \II ,\,B_Q\cap B(x,r)\neq\varnothing\}\leq \ve^{1/2}r.
\end{equation}
Let $L$ be the hyperplane parallel to $\PP(x',r)$ that contains $x$. 
Then we set
\begin{align*}
\dist_H(\partial\Omega_\ve^+\cap B(x,r), \,L\cap B(x,r)) &\leq
\dist_H(\partial\Omega_\ve^+\cap B(x,r), \,\partial\Omega\cap B(x,r))\\
&\quad + 
\dist_H(\partial\Omega\cap B(x,r), \,\PP(x',r)\cap B(x,r))\\
& \quad + 
\dist_H(\PP(x',r)\cap B(x,r), \,L\cap B(x,r))\\
& = \circled{1} + \circled{2} + \circled{3}.
\end{align*}
It is immediate to check that
$$\circled{3} \leq c\,|x-x'|\leq c\,\ve^{1/2}r.$$
On the other hand, to estimate \circled{2} we write
\begin{align*}
\circled{2} & \leq 
\dist_H(\partial\Omega\cap B(x,r), \,\partial\Omega\cap B(x',r))
+
\dist_H(\partial\Omega\cap B(x',r), \,\PP(x',r)\cap B(x',r))\\
&\quad
+
\dist_H(\PP(x',r)\cap B(x',r), \,\PP(x',r)\cap B(x,r)).
\end{align*}
It is easy to check that the first and the last terms on the right hand side above do not
exceed $c\,|x-x'|\leq c\,\ve^{1/2}r$, while 
$$\dist_H(\partial\Omega\cap B(x',r), \,\PP(x',r) \cap B(x',r))\leq \delta\,r,$$
by the Reifenberg flatness of $\Omega$. Thus,
$$\circled{2}\lesssim (\ve^{1/2} + \delta)r \lesssim \ve^{1/2}r,$$
assuming $\delta\leq\ve^{1/2}$.

To estimate \circled{1} we use the fact that for every $y\in\partial \Omega^+_\ve\cap B(x,r)$ there exists
$y'\in\partial\Omega$ such that $|y-y'|\leq\ve^{1/2}r$, and also for every $z'\in\partial\Omega\cap B(x,r)$
there exists some $z\in\partial\Omega_\ve^+$ such that $|z-z'|\leq c\,\ve^{1/2}r$. The existence of
$y'$ follows by arguing as in \rf{eq101}, while the existence of $z$ can be shown with the
aid of Lemma \ref{lemreif2}. From these facts one can derive that
$$\circled{1}\lesssim \ve^{1/2}r.$$
If we gather the estimates for \circled{1}, \circled{2} and \circled{3}, we get
%$$\dist_H(\partial\Omega_\ve^+\cap B(x,r), \,\PP(x,r)\cap B(x,r))\lesssim \ve^{1/2}r.$$
$$\dist_H(\partial\Omega_\ve^+\cap B(x,r), \, L\cap B(x,r))\lesssim \ve^{1/2}r.$$

We claim now that one of the connected components of 
$$B(x,r)\cap \bigl\{y\in\R^{d+1}:\dist(y,L) \geq c\,\ve^{1/2}\,r\bigr\}$$
is contained in $\Omega_\ve^+$ and the other is contained in $\R^{d+1}\setminus\Omega_\ve^+$.
To see this, pick $x'\in \d\Omega$ closest to $x$. We take into account that, for $\delta$ small enough, and since $\dist(L,P(x',r))=|x-x'|\leq \ve^{1/2}r$,
$$U_{2\delta r}(\PP(x',r))\subset U_{c\ve^{1/2}r}(L),$$
and thus by the property (b) in the Definition \ref{defreif} applied to $\Omega$, $x'$ and $2r$,
one of the components of $B(x,\frac32r)\setminus U_{c\ve^{1/2}r}(L)$ is contained in $\Omega$ and the other
in $\R^{d+1}\setminus\Omega$. Since all the balls $B_Q$ which intersect $B(x,r)$ have radius at most $\ve^{1/2}r$ and they intersect $\partial\Omega$, we infer that all of them are contained in
$U_{c'\ve^{1/2}r}(L)$, and then 
$$\bigl(B(x,r)\setminus U_{c'\ve^{1/2}r}(L)\bigl)\cap \Omega = 
\bigl(B(x,r)\setminus U_{c'\ve^{1/2}r}(L)\bigl)\cap \Omega_\ve^+$$ 
and
$$\bigl(B(x,r)\setminus U_{c'\ve^{1/2}r}(L)\bigl)\cap \Omega^c = 
\bigl(B(x,r)\setminus U_{c'\ve^{1/2}r}(L)\bigl)\cap (\Omega_\ve^+)^c,$$
which implies that (b) in the Definition \ref{defreif} holds for $\Omega_\ve^+$, $x$ and $r$.

\vv
\noi{\bf Case 2.} Suppose that $r_Q< r <\ve^{-1/2}r_Q$ for some $Q\in \II $ such that $B_Q\cap B(x,2r)\neq\varnothing$.
For this case, we will require the following lemma that will shorten some computations.

\begin{lemma}\label{otherdistance}
For a closed set $E$, $x\in E$, $r>0$, and a $d$-plane $P$ intersecting $B(x,r)$, we have
\[\max\left\{\sup_{y\in E\cap B(x,r)} \dist(y,P),\sup_{y\in P\cap B(x,r)}\dist(y,E)\right\} \sim \dist_{H}\bigl(E\cap B(x,r),P\cap B(x,r)\bigr).\]
\end{lemma}
We leave the details to the reader.

%Suppose that $r_Q< r <\ve^{-1/2}r_Q$ for some $Q\in \II $ such that $B_Q\cap B(x,2r)\neq\varnothing$. 
We denote
$$\wt L_Q= L_Q + r_Q\,N_{z_Q,r_Q}$$
where again $L_Q=\PP(z_Q,30r_Q)$. We will show that, for the point $x\in \d\Omega_{\ve}^{+}$,
\begin{equation}\label{eq*0}
\dist_H(\partial\Omega^+_\ve\cap B(x,r), \wt L_Q\cap B(x,r))\lesssim\ve^{1/2}\,r.
\end{equation}
Although we cannot guarantee that $x\in \wt L_Q$, it is clear that from this estimate one deduces
that (a) from Definition \ref{defreif} holds just by translating $\wt L_Q$ appropriately.

To prove \rf{eq*0} first we claim that if $B_P\cap B(x,3r)\neq\varnothing$ for some $P\in \II $, then $\ell(P)\sim\ell(Q)$ and
\begin{equation}\label{eq*00}
|r_P-r_Q|\lesssim \ve^{1/2}r_Q.
\end{equation}
To see this, notice that
$$\dist(B_P,B_Q)\leq 5r\leq 5\ve^{-1/2} \,r_Q.$$
Then, recalling that $r_P\leq \ve\,\dist(z_P,E)$ for all $P\in \II $, we get 
\begin{equation}%\label{c}
\label{eqsup22}
|r_Q - r_P|\leq \ve\, |z_P - z_Q|\leq \ve(r_P+ r_Q + 5\ve^{-1/2} \,r_Q) \lesssim \ve^{1/2}(r_Q+r_P).
\end{equation}
In particular, this implies that $r_Q\sim r_P$ and so that $\ell(P)\sim \ell(Q)$ and finishes the proof of the claim.
%For the record, notice that to obtain \rf{eqsup22} we only have taken into account that $r <\ve^{-1/4}r_Q$. That is, we have not used the fact that $ r>\ve^{1/4}r_Q$.

We will need the following well known lemma, a proof of which is supplied in \cite{AT}.
\begin{lemma}\label{angles}
Suppose $P_{1}$ and $P_{2}$ are $n$-planes in $\R^{d+1}$ and $X=\{x_{0},...,x_{n}\}$ are points so that
\begin{enumerate}
\item[(a)] $\eta=\eta(X)=\min_{i} \dist(x_{i},\textrm{span} (X\backslash\{x_{i}\}))/\diam X\in (0,1)$ and
\item[(b)] $\dist(x_{i},P_{j})<\theta\,\diam X$ for $i=0,...,n$ and $j=1,2$, where $\theta<\eta (d+1)^{-1}/2$.
\end{enumerate}
Then
\begin{equation}
\dist(y,P_{1}) \leq \theta\left(\frac{2d}{\eta}\dist(y,X)+\diam X\right) \mbox{ for all  }y\in P_{2}.
\end{equation}
\end{lemma}

We continue now with the proof of \rf{eq*0}. To this end, denote $L=\PP(z_Q,30r)$ and assume 
\begin{equation}\label{bpbx3r}
B_{P}\cap B(x,3r)\neq\emptyset.
\end{equation}
 Let $x_{0}=z_{P}$ and $x_{1},...,x_{d}\in L_{P}$ be such that $|x_{i}-x_{0}|=r_{P}$ for all $i=1,...,d$ and $X=\{x_{0},..,x_{d}\}$ is a scaled copy of the vectors $\{e_{0},...,e_{d}\}\subset \mathbb{R}^{d}$ where $e_{0}=0$ and $e_{1},...,e_{d}$ are the standard basis vectors. Then it is not hard to show that $\eta(\{x_{0},...,x_{d}\})=\eta(\{e_{0},...,e_{d}\})\sim_{d} 1$ and $\diam X\sim_{d} r_{P}$. 
 %Since $r_{P}<2r_{Q}<2r$ by \eqref{eq*00} for $\ve$ small enough, $B_{P}\cap B(x,3r)\neq\emptyset$, and $B(x,2r)\cap B_{Q}\neq\emptyset$. Also,
 By the definition of $L_{P}$, there are $x_{i}'\in \d\Omega$ with $|x_{i}-x_{i}'|<30\delta r_P\leq 30\delta r$, and so for each $i=0,...,d$, by \eqref{bqbx2r} and \eqref{bpbx3r},
\[ x_{i}'\in B(z_{P},r_{P}+30\delta r_{P})\subseteq B(x,3r+30\delta r)
\subseteq B(z_{Q},r_{Q}+2r+3r+30\delta r)\subseteq B(z_{Q},7r)\]
Thus, by the definition of $L$, there are $x_{i}''\in L$ so that $|x_{i}'-x_{i}''|< 30\delta r$, hence 
\[
\dist(x_{i},L)\leq |x_{i}-x_{i}''|<60\delta r\lesssim \ve^{-1/2}\delta r_{P}\sim \delta \ve^{-1/2}\diam X\]
%
% Moreover, since $r_{P}<2r_{Q}<2r$ by \eqref{eq*00} for $\ve$ small enough, $B_{P}\cap B(x,3r)\neq\emptyset$, and $B(x,2r)\cap B_{Q}\neq\emptyset$, we have 
%\[ B_{P}\subset B(x, 3r+2r_{P})\subset B(x,3r+4r_{Q})\subset B(z_{Q},6r+4r_{Q})\subset B(z_{Q},10r)\]
%and so by the definition of $L$, there are points $x_{i}'\in \d\Omega$ with $|x_{i}-x_{i}'|<30\delta r$, and so $x_{i}'\in B(z_{P},r_{P}+30\delta r)\subset B(z_{P},2r_{P})\subset B(z_{P},30r_{P})$ for $\delta\ll \ve^{\frac{1}{2}}$. Thus, there are points $x_{i}''\in L_{P}$ with $|x_{i}'-x_{i}''|<30\delta r_{P}$, and thus 
%\[\dist(x_{i},L_{P})\leq |x_{i}-x_{i}''|\lesssim \delta r + \delta r_{P} \lesssim \delta \ve^{-1/2} r_{P}
%\sim_{d}  \delta\ve^{-\frac{1}{2}}\diam X\]
and thus the previous lemma implies, for $B_P\cap B(x,3r)
\neq\varnothing$,
\begin{equation}\label{eqclaim**}
\dist_H(L\cap B(x,30r),\,L_P\cap B(x,30r))\leq  \, c\delta\,\ve^{-1/2}\,r<\ve r.
\end{equation}

Then, assuming $\delta\ll\ve$, we deduce that
$$\dist_H(L\cap 10B_P,\,L_P\cap 10B_P)\leq 
\dist_H(L\cap B(x,30r),\,L_P\cap B(x,30r))\leq \ve\,r.$$
In particular, this also holds for $P=Q$, and thus we have
\begin{align}\label{eqdif115}
\dist_H(L_Q\cap 10B_P,\,L_P\cap 10B_P)&\le \dist_H(L_Q\cap 10B_P,\,L\cap 10B_P)\\&\quad
+ \dist_H(L\cap 10B_P,\,L_P\cap 10B_P)\nonumber\\
&\leq 2\ve\,r\lesssim\ve^{1/2}r_P.\nonumber
\end{align}

Given $P$ as above, we consider the hyperplanes
$$%\wt L = L + r_Q\,N_{z_Q,r_Q},\qquad
\wt L_P= L_P + r_P\,N_{z_P,r_P},\qquad \wt L_{Q,P} = L_Q + r_P\,N_{z_P,r_P}.$$
Notice that, by Lemmas \ref{lemreif2} and \ref{otherdistance},
\begin{equation}\label{eqre32}
\dist_H(\wt L_P\cap10B_P, \partial\Omega^+_\ve\cap10B_P)\lesssim \ve\,r_P.
\end{equation}
 We write
\begin{align*}
\dist_H(\wt L_Q\cap10B_P,\,\wt L_P\cap 10B_P) & \leq 
\dist_H(\wt L_Q\cap 10B_P,\,\wt L_{Q,P}\cap 10B_P)\\
&\quad +\dist_H(\wt L_{Q,P}\cap 10B_P,\,\wt L_P\cap10B_P).
\end{align*}
From \rf{eqdif115}, it follows easily that $|N_{z_P,r_P} - N_{z_Q,r_Q}|\lesssim \ve^{1/2}$.
Taking also into account
that $|r_P-r_Q|\lesssim \ve^{1/2}r_Q\le \ve^{1/2}r$ by \eqref{eq*00}, 
it is easy to check that
$$\dist_H(\wt L_Q\cap 10B_P,\,\wt L_{Q,P}\cap 10B_P)\lesssim |r_P\,N_{z_P,r_P} - r_Q\,N_{z_Q,r_Q}|
\lesssim \ve^{1/2}\,r_Q + \ve^{1/2}\,r\lesssim \ve^{1/2}r$$
and, by \rf{eqdif115},
$$\dist_H(\wt L_{Q,P}\cap 10B_P,\wt L_P\cap 10B_P)\lesssim \dist_H(L_Q\cap 10B_P,L_P\cap 10B_P)
\lesssim \ve\,r.$$
So we deduce from the above two inequalities that
$$\dist_H(\wt L_Q\cap 10B_P,\,\wt L_P\cap 10B_P) \lesssim \ve^{1/2}r.$$
Together with \rf{eqre32}, this gives
\begin{equation}\label{eq*0.5}
\dist_H(\partial\Omega^+_\ve\cap10B_P, \wt L_Q\cap 10B_P)\lesssim \ve^{1/2}r + \ve\,r_P
\lesssim \ve^{1/2}r.
\end{equation}

To see that \rf{eq*0} holds, by Lemma \ref{otherdistance} it suffices to show that 
\begin{equation}\label{eq*1}
\mbox{
for all 
$y\in \partial \Omega^+_\ve\cap B(x,r)$ there exists $y'\in \wt L_Q$ such that $|y-y'|\lesssim\ve^{1/2}r$,}
\end{equation}
and
\begin{equation}\label{eq*2}
\mbox{
for all 
$y\in \wt L_Q\cap B(x,r)$ there exists $y'\in \partial  \Omega^+_\ve$ such that $|y-y'|\lesssim\ve^{1/2}r$.}
\end{equation}
The statement \rf{eq*1} holds because of \rf{eq*0.5} and the fact that 
\begin{equation}\label{locontain}
\partial\Omega_\ve^+\cap B(x,r)\subset \bigcup_{P\in \II :P\cap B(x,r)\neq \varnothing} B_P
\end{equation}
Indeed, it is easy to see that, under our current assumptions, $x \in E^c$ and by virtue of \rf{eq*00} we have that $\partial\Omega_\ve^+\cap B(x,r) \subset E^c$. To prove \rf{eq*2}, given $y\in\wt L_Q\cap B(x,r)$, let $z$ be the orthogonal
projection of $y$ on $L_Q$, so that $|y-z|=r_Q$, hence $z\in B(x,r+r_Q)$. By  \eqref{eqclaim**} (with $P=Q$ and $\ve$ small enough), we may find $\wt z\in B(x,30r)\cap L$ and $|\wt z-z|<\ve r$, and by the definition of $L$ and for $\delta$ small, we can find $z'\in \d\Omega$ with $|z'-\tilde{z}|<30\delta r<\ve r$, and so $|z-z'|<2\ve r$.
Let $P\in \II $ be such that $z'\in B_P$. Notice that
$$|z'-x|\leq |z'-z|+ |z-x|\leq  2\ve r+r+r_Q \leq 3r.$$
So $B_P\cap B(x,3r)\neq\varnothing$. Also, $y\in 10B_P$ because
$$|y-z'|\leq |y-z|+|z-z'|\leq r_Q + 2\ve\,r\leq r_Q + 2\ve^{1/2}r_Q \leq 2r_P,$$
recalling \rf{eqsup22}.
Hence, from \rf{eq*0.5} we infer that there exists $y'\in 10B_P\cap\partial\Omega_\ve^+$ 
such that $|y-y'|\lesssim \ve^{1/2}r$. So the proof of \rf{eq*0} is concluded.

The condition (b) in Definition \ref{defreif} can be shown as in case 1, we omit the details.

\vv
\noi{\bf Case 3.} Suppose that $r \leq r_Q$ for some $Q\in \II $ such that $B_Q\cap B(x,2r)\neq\varnothing$.

In this case $B(x,r)\subset 10B_Q$. Recalling that $\partial\Omega_\ve^+$ coincides on $10 B_Q$ with the graph of a Lipschitz function $f_Q:L_Q\to L_Q^\bot$ with Lipschitz constant not exceeding $c\ve^{1/2}$, if we denote by $L_{Q,x}$ the hyperplane that is parallel to $L_Q$ and passes through $x$, we get by Lemmas \ref{lemreif2} and \ref{otherdistance}
$$\dist_H(\partial\Omega_\ve^+\cap B(x,r),\,L_{Q,x}\cap B(x,r))\lesssim \ve^{1/2}\,r.$$

Again, condition (b) in Definition \ref{defreif} can be shown as in case 1, we omit the details.
\end{proof}

\section{Radon measures of low dimension}

Some of our work toward the main result can be done in more generality than with harmonic measure. We will apply the following theorem in the last section with $\mu$ equal to the harmonic measure for a suitable modification of a Wolff snowflake domain.

\begin{theorem}\label{thm:finitelength}
Let $\Omega$ be a $(\delta,r_{0})$-Reifenberg flat domain, $\xi_{0}\in \d\Omega$, and $E\subset B(\xi_{0},r_{0}) \cap \d \com$ a closed set. Also assume that there is a Radon measure $\mu$ with support contained in $\d\Omega$ such that $\mu(B(\xi,r))\geq c_{\mu}r^{d-\alpha}$ for all $\xi\in E$, $r<r_{0}$ and some constants $c_{\mu},\alpha>0$. If $\Omega_{\varepsilon}^{+}\supset\Omega$ is the domain from  Lemma \ref{lemreif1}, then $\cH^{d}|_{\d\Omega_{\varepsilon}^{+}}$ is locally finite.
\end{theorem}

\begin{proof}
For $Q\in \mathcal I$, let $\Gamma_{Q}=10B_{Q}\cap \d\Omega_{\varepsilon}^{+}$, so by Lemma \ref{lemreif2}, this is a Lipschitz graph.

Let $\xi\in \d\Omega_{\varepsilon}^{+}$, and $r>0$ be much smaller than $r_{0}$ (how small will depend on $\ve$ and $d$), and define 
\[C(\xi,r)=\{Q\in \II:\Gamma_{Q}\cap B(\xi,r)\neq\emptyset\}.\] 

Our goal now is to show that $\cH^{d}(\d\Omega_{\varepsilon}^{+}\cap B(\xi,r))<\infty$ for all $r\in(0,r_{0})$. We consider two cases.

\vv
\noi{\bf Case 1.} If $\dist(\xi,E)\geq 2r$, then all cubes $Q\in C(\xi,r)$ have comparable sidelengths. To see this, suppose we can find $Q_{j}\in C(\xi,r)$ so that $\ell(Q_{j})\rightarrow 0$. Then eventually, $\ell(Q)\ll r_{0}$ so that $\ell(Q_{j})\sim \ve^{2}\dist(x,E)$, thus
\[2r\leq \dist(\xi,E)\leq \liminf (\dist(\xi,Q_{j})+\diam Q_{j})\leq r+0,\]
which gives a contradiction. Thus, $\inf\{\ell(Q):Q\in C(\xi,r)\}>0$. On the other hand, for all $Q\in C(\xi,r)$, 
\[\ell(Q)\lesssim \ve^{2}\dist(Q,E)\leq \ve^{2}(r+\dist(\xi,E))<\infty.\]
Hence, all $Q\in C(\xi,r)$ have comparable diameters. Since they all intersect $B(\xi,r)$, this means there must be finitely many of them, and thus $\cH^{d}(B(\xi,r)\cap \d\Omega_{\varepsilon^{+}})<\infty$. 

%
%then 
%\[ B(\xi,r)\cap \d\Omega_{\varepsilon}^{+}\subset \bigcup_{Q\in C(\xi,r)} \Gamma_{Q}.\]
%If $Q\in C(\xi,r)$, let $y_{Q}\in \Gamma_{Q}\cap B(\xi,r)$. By part (b) of Lemma \ref{lemreif2},
%\begin{equation}\label{e:QGQ}
%\dist(Q,\Gamma_{Q})
%\leq \dist(z_{Q},\Gamma_{Q})
%\lesssim r(B_{Q})
%\sim \ve^{-1}\ell(Q)
%\end{equation}
%so this and \eqref{e:gammavol} imply
%\[|z_{Q}-y_{Q}|\leq \dist(z_{Q},\Gamma_{Q})+\diam \Gamma_{Q}\lesssim \ve^{-1} \ell(Q).\]
%Thus, since $\dist(y_{Q},E)\geq \dist(\xi,E)-|\xi-y_Q|\geq r$, we have
%\[
%\varepsilon^{-2}\ell(Q)\sim \dist(Q,E)\sim \dist(z_{Q},E)
%\geq  \dist(y_{Q},E)-|y_{Q}-z_{Q}| 
%\geq r-C\varepsilon^{-1} \ell(Q)\]
%for some constant $C$ depending only on $d$. Thus, for $\varepsilon$ small enough, $\ell(Q)\gtrsim \varepsilon^{2} r$. Since the cubes in $C(\xi,r)$ are disjoint (by virtue of being Whitney cubes) and with side lengths at least a multiple of $\varepsilon^{2}r$, there are at most $N_{0}=N_{0}(\varepsilon,d)$ many of them. Finally, for each such $Q$, by \eqref{e:gammavol} and the preceding discussion, $\cH^{d}(\Gamma_{Q})\sim \ve^{-d}\ell(Q)^{d}\sim \ve^{d} r^{d}$, and thus
%\[ \cH^{d}(B(\xi,r)\cap \d\Omega_{\varepsilon}^{+})
%\leq \sum_{Q\in C(\xi,r)} \cH^{d}(\Gamma_{Q})
%\lesssim_{N_{0},\ve} r^{d}.\]

\vv
\noi{\bf Case 2.} Now suppose $\dist(\xi,E)<2r$. Note that 
\begin{align*}
\dist(Q,E)& \lesssim \dist(Q,\Gamma_{Q})+\diam \Gamma_{Q}+\dist(\Gamma_{Q},\xi)+\dist(\xi,E)\\
& \lesssim r_{Q}+r+2r \lesssim \ve^{-1} \ell(Q) 
\lesssim \ve\dist(Q,E)+r
\end{align*}
and so for $\ve$ small enough $\dist(Q,E)\lesssim r$, and so $\dist(Q,E)<r_{0}$ if $r\ll r_{0}$, thus 
\[\ell(Q)\sim \ve^{2} \dist(Q,E) \lesssim \ve^{2}r\] 
for all $Q\in C(\xi,r)$. For $Q\in C(\xi,r)$, pick $\xi_{Q}\in E$ so that $\dist(\xi_{Q},Q)=\dist(Q,E)\sim \ve^{-2}\ell(Q)$ and let $B^{Q}=B(\xi_{Q},\ell(Q))$. For $n\in \mathbb{Z}$, define
\[C_{n}(\xi,r)=\{Q\in C(\xi,r): \ell(Q)=2^{-n}\}.\]

We claim that there is $N_{1}=N_{1}(\ve,d)$ so that no point in $E$ is contained in more than $N_{1}$ many $B^{Q}$ with $Q\in C_{n}(\xi,r)$. Thus, fix $n\in \mathbb{Z}$ and $\zeta\in E$. If $Q\in C_{n}(\xi,r)$ is such that $\zeta\in B^{Q}$, then
\begin{align*}
\dist(\zeta,Q)
& \leq |\zeta-\xi_{Q}|+\dist(\xi_{Q},Q)
<c\,\ell(Q)+\dist(Q,E)
 \sim \varepsilon^{-2}\ell(Q) =\varepsilon^{-2}2^{-n}.
\end{align*}
Thus, all cubes $Q\in C_{n}(\xi,r)$ for which $\zeta\in B^{Q}$ are disjoint, contained in a ball of radius $C2^{-n}$ for some $C=C(\varepsilon,d)$, and are of side length $2^{-n}$, so there can only be at most $N_{1}=N_{1}(\varepsilon,d)$ many of them, which settles the claim. 

% For $Q\in C(\xi,r)$, and since $\dist(Q,\Gamma_{Q})\lesssim r(B_{Q})\sim \ve^{-1}\ell(Q)$, we get
%\begin{align*}
%\varepsilon^{-2}\ell(Q)\sim \dist(Q,E)
%& \leq
%\dist(Q,\Gamma_{Q})+\diam \Gamma_{Q}+\dist(\Gamma_{Q},\xi)+\dist(\xi,E)\\
%& \lesssim  \varepsilon^{-1}\ell(Q)+\varepsilon^{-1}\ell(Q)+r+2r
% \lesssim \varepsilon^{-1}\ell(Q)+r
%\end{align*}
%and so for $\varepsilon$ small enough, $\ell(Q)\lesssim \varepsilon^{2}r$, and again for $\ve$ small enough, $\ell(Q)<r<r_{0}$ since $r<r_{0}$. 

Since $\ell(Q)\lesssim \varepsilon^{2}r$, and again for $r\ll r_{0}$, $\ell(Q)<r<r_{0}$. In particular, $C_{n}(\xi,r)\neq\emptyset$ implies $2^{-n}\leq r$. Thus, since $\diam \Gamma_{Q}\sim r_{Q}\sim \ve^{-1} \ell(Q)$,
\begin{align*} 
|\xi_{Q}-\xi|
& \leq \dist(\xi_{Q},Q)+\diam Q+\dist(Q,\Gamma_{Q})+\diam \Gamma_{Q}+\dist(\xi,\Gamma_{Q})\\
& \lesssim_{d} \ve^{-2}\ell(Q)+\ell(Q)+\ve^{-1}\ell(Q)+\ve^{-1}\ell(Q)+r
 \lesssim r.
\end{align*}
Thus, $Q\in C(\xi,r)$ implies $\xi_{Q}\in B(\xi,(C-1)r)$ for some large $C>0$, depending on $d$, and $\ell(Q)<r$ implies $B^{Q}\subset B(\xi,Cr)$. Finally, note that $Q\in C_{n}(\xi,r)$ implies
\[\cH^{d}(\Gamma_{Q})
\sim \varepsilon^{-d}\ell(Q)^{d}
\lesssim_{\varepsilon,d} \ell(Q)^{\alpha} r(B^{Q})^{d-\alpha}
\leq c_{\mu}^{-1} 2^{-n\alpha}\mu(B^{Q}).\]
Therefore,
\begin{align*}
\cH^{d}(\d\Omega_{\varepsilon}^{+}\cap B(\xi,r))
& \leq \sum_{Q\in C(\xi,r)}\cH^{d}(\Gamma_{Q})+\cH^{d}(E)
=\sum_{2^{-n}\leq r}\sum_{Q\in C_{n}(\xi,r) } \cH^{d}(\Gamma_{Q})+0\\
& \lesssim \sum_{2^{-n}\leq r}\sum_{Q\in C_{n}(\xi,r) } \frac{\mu(B^{Q})}{2^{n\alpha}c_{\mu}}
 \leq \frac{N_{1}}{c_{\mu}}\sum_{2^{-n}\leq r} 2^{-n\alpha}\mu(B(\xi,Cr))\\
 & \lesssim_{\ve,d,c_{\mu}} r^{\alpha}\mu(B(\xi,Cr))<\infty
\end{align*}

The proof of the theorem is finished now that we have shown these two cases.

\end{proof}

% ***************************************************************************

\section{Wolff snowflakes and harmonic measure}

We will now describe the construction of the Wolff snowflake domain. We follow closely the
approach of \cite{LNV}, which in turn is just a small variant of the original construction of 
Wolff in \cite{Wolff}. We also remark that the below description of the construction of the Wolff snowflakes is an almost verbatim copy of an analogous one in \cite{LNV}.
For more details, see the aforementioned references.

Let $\com_0=\{(x',x_{d+1}): x' \in \R^d, x_{d+1}>0\}$ and set
$Q(r):=\{x' \in \R^d: -r/2 \leq x_i \leq r/2, \,\textup{for} \, 1 \leq i \leq d\}$. Then $Q(r)$ is a $d$-dimensional cube with side length $r$ and center $0$. Let $\phi:\R^d \to \R$ be a piecewise
linear function with $\supp(\phi) \subset \{x' \in \R^d: |x'| <1/2$ and $\|\nabla \phi\|_\infty \leq \theta$. For fixed $N$ large, set $\psi(x')= N^{-1} \phi(N x')$. Let $b>0$ be a small constant and let $Q$	 be a $d$-cube (i.e.,
a $d$-dimensional cube contained in some hyperplane) with center $a_Q$ and side length
$\ell(Q)$. Let $e$ be a unit normal to $Q$ and define
$$P_Q=\text{cch}(Q \cup \{a_Q +b \ell(Q) e\}), \quad  \widetilde P_Q=\text{int cch}(Q \cup \{a_Q -b \ell(Q) e\}),$$
where cch$E$ and int$E$ denote the closed convex hull and interior of $E$, respectively. For
the cube $Q(1)$ set $e=-e_n=(0,\dots,0,-1)$ and let
\[\Lambda=\{x\in P_{Q(1)} \cup \widetilde P_{Q(1)} :x_{d+1}>\psi(x)\}\]
\[ \partial=\{x\in \R^{d+1}: x' \in Q(1),\, x_{d+1}=\psi(x')\}.\]

We assume that $N=N(b, \theta)$  is so large that $\dist \left (\partial \setminus \partial \Omega_0, \,\partial[P_{Q(1)} \cup \widetilde P_{Q(1)}] \right) \geq b/100$. Note that $\partial \subset Q(1) \times [-1/2,1/2]$ consists of a finite number of $d$-dimensional faces. We fix a Whitney decomposition of each face. That is, we divide each face of $\partial$ into $d$-cubes $Q$, with side lengths $8^{-k}, k=1, 2, \dots$ which are proportional to their distance
from the edges of the face they lie on. We also choose a distinguished $(d-1)$-dimensional
``side" for each $d$-cube.

Suppose $\com$ is a domain and $Q \subset \partial \com$ is a $d$-cube with distinguished side $\gamma$. Let $e$ be the outer unit normal to $\partial \Omega$ on $Q$ and suppose that $P_Q \cap \com=\emptyset$ and $\widetilde P_Q \subset \Omega$. We form
a new domain $\widetilde \com$ as follows. Let $T$ be the conformal affine map (i.e., a composition of a
translation, rotation, dilation) with $T(Q(1))=Q$ which fixes the dilation, $T(0)=a_Q$
which fixes the translation, and finally fix the rotation by requiring that $T(\{x \in \partial Q(1): x_1=1/2\})=\gamma$ and $T(-e_n)$ is in the direction of $e$. Let $\Lambda_Q= T(\Lambda)$ and $\partial_Q=T(\partial)$. Then we define $\widetilde \com$ through the relations $\widetilde \com \cap (P_Q \cup \widetilde P_Q)=\Lambda_Q$ and $\widetilde \com \setminus  (P_Q \cup \widetilde P_Q)=\com \setminus  (P_Q \cup \widetilde P_Q)$. Note that $\partial_Q$ inherits from $\partial$ a natural subdivision into Whitney cubes with distinguished sides. We call this process ``adding a blip to $\com$ along $Q$".

To use the process of ``adding a blip" to construct a Wolff snowflake $\Omega_\infty$, starting
from $\com_0$, we first add a blip to $\com_0$ along $Q(1)$ obtaining a new domain $\com_1$. We then inherit a subdivision of $\partial \com_1 \cap ( P_{Q(1)} \cup \widetilde P_{Q(1)} )$ into Whitney cubes with distinguished sides, together with a finite set of edges $E_1$ (the edges of the faces of the graph are not in the Whitney cubes). Let $G_1$ be the set of all Whitney cubes in the subdivision. Then $\com_2$ is obtained from $\com_1$ by adding a blip along each $Q \in G_1$. From this process, we inherit a family of cubes $G_2 \subset \partial \com_2$ (each with a distinguished side) and a set of edges $E_2 \subset \partial \com_2$ of $\sigma$-finite $\cH^{d-1}$-measure. Continuing by induction we get $(\com_m)_{m=1}^\infty$, $(G_m)_{m=1}^\infty$ and $(E_m)_{m=1}^\infty$, where $\partial \com_m \cap ( P_{Q(1)} \cup \widetilde P_{Q(1)} ) = E_m \cup \bigcup_{Q \in G_m} Q$ for $m \geq 1$. If $N=N(b, \theta)$ is large enough, then $\dist_H(\com_m,\com_\infty) \to 0$ as $m \to \infty$. We call $\com_\infty$ a $\theta$-\textbf{Wolff snowflake} domain.

The following result is proved in \cite[Lemma 7.1]{LNV}.

\begin{lemma}\label{lem:Wolff-Rflat}
Suppose that $\theta \in (0,1)$ is small enough and $N$ large enough, depending on $d,b,$ and $\theta$. Then the $\theta$-Wolff snowflake domain $\com_\infty$ is $(c_1\,\theta,\infty)$-Reifenberg flat, for some positive constant $c_1$.
\end{lemma}

\begin{rem}
 Similarly, for any fixed $\varepsilon>0$ and $\tau>0$, one may construct a {\it bounded} $\theta$-Wolff snowflake domain (see also \cite{LNV} and \cite{Wolff}). Indeed, this is done by taking the unit cube in $\mathbb{R}^{d+1}$ contained in the lower half-space that has $Q(1)$ as one of its faces (its "bottom" face). Then we just mimic the construction above to each face of the cube. We will denote this new domain by $\widetilde{\Omega}_{\infty}$. Notice here that $\widetilde{\Omega}_{\infty}\subset\Omega_{\infty}$ and
$$\d\widetilde{\Omega}_{\infty}\cap \{x\in \mathbb{R}^{d+1}:x_{d+1} <0\}=  \d\Omega_{\infty}\cap  \{x\in \mathbb{R}^{d+1}:x_{d+1}<0\}.$$
\end{rem}

Before we apply our results from the previous section let us introduce some notation. If $\mu$ is a Borel probability measure in $\mathbb{R}^{d+1}$, we define its \textit{lower pointwise dimension} at the point $x \in \supp \mu$ to be 
$$\underline d_\mu(x) = \liminf_{r \to 0} \frac{\log \mu(B(x,r))}{\log r}$$ 
and its \textit{upper pointwise dimension} at the point $x \in \supp \mu$ 
$$\overline d_\mu(x) = \limsup_{r \to 0} \frac{\log \mu(B(x,r))}{\log r}.$$ 
The common value $\underline d_\mu(x) =\overline d_\mu(x)=d_\mu(x)$, if it exists, we call it \textit{pointwise dimension} of $\mu$ at $x \in \supp \mu$.

Let $\dim_H(Z)$ be the {\it Hausdorff dimension of the set $Z$}. Given a measure $\mu$ on a set $\Lambda \subset \R^{d+1}$ the {\it Hausdorff dimension of $\mu$} is defined by
$$
\dim_H(\mu)= \inf\{ \dim_H(Z): Z \subset \Lambda \,\, \text{and} \,\, \mu(\Lambda \setminus Z)=0\}.
$$
Moreover,
$$
\dim_H(\mu)= \esssup \{\underline d_\mu(x): x \in \Lambda\},
$$
where the essential supremum is taken with respect to $\mu$ (see Proposition 3, \cite{BW}). In particular, if there exists a number $\delta$ so that $d_\mu(x)=\delta$ for $\mu$-a.e. $x \in \Lambda$, then $\dim_H(\mu)=\delta$. This criterion was established in \cite{Young} by Young.

The following theorem was  proved by Wolff \cite{Wolff} %and Lewis, Nystr\"om and Vogel \cite{LNV} respectively. 
\begin{theorem}\label{lem:Wolff}
For every $\theta>0$ there exists a bounded $\theta$-Wolff snowflake domain $\widetilde \com_\infty \subset \R^{d+1}$ built from a Lipschitz function function $\phi$ with Lipschitz constant at most $\theta$ such that $$d_{\om_{\widetilde \com_\infty}}(x) =s<d, \quad \text{for all}  \,\, x \in \d \widetilde \com_\infty \setminus \mathcal X,$$  where $\mathcal X \subset \widetilde \com_\infty$ is such that $\om_{\widetilde \com_\infty}(\mathcal X)=0$ and $ \om_{\widetilde \com_\infty}$ is the harmonic measure in $\widetilde \com_\infty$.
\end{theorem}

%\begin{lemma}\label{lem:LNV}
%Let $p>2$. Let also $\com_\infty \subset \R^{d+1}$ be a Wolff snowflake domain, $\Theta_{Q(1)}:=( P_{Q(1)} \cup \widetilde P_{Q(1)} )  \cap \d\com_\infty$ and $\widetilde\om_p$ the restriction of the $p$-harmonic measure on the set $\Theta_{Q(1)}$. Then we have that $d_{\widetilde \om_p}(x) < d$ for $\widetilde \om_p$-a.e. $x \in \Theta_{Q(1)}$.
%\end{lemma}

We shall use now Lemma \ref{lem:Wolff-Rflat} and Theorem \ref{lem:Wolff} %and \ref{lem:LNV}
in order to apply the results from the previous sections and obtain our main theorem.

Note that from now we identify $\{(\tilde x, x_{d+1})\in \R^{d+1}: x_{d+1}=0\}$ with $\R^d$. 

\vspace{3mm}
 
\begin{proof}[Proof of Theorem \ref{maintheorem}]

Let $\theta>0$ be a sufficiently small constant that will be chosen momentarily and assume that $\widetilde \com_\infty$ is a $\theta$-Wolff snowflake domain. Let us also fix a pole $z_0\in \widetilde \com_\infty$. 

By Theorem \ref{lem:Wolff} we have that $ d_{\om_ {{\widetilde\com}_\infty}} (\xi) \leq s< d$ for $\om_ {{\widetilde\com}_\infty}$-a.e. $\xi \in \partial \widetilde \Omega_\infty$ and thus, $\dim_H(\om_ {{\widetilde\com}_\infty})\leq s<d$. This implies that there exists a set $X \subset  \partial \widetilde\Omega_\infty$  so that $\om_ {{\widetilde\com}_\infty}(X)=1$, $\dim_H(X)  \leq s < d$ (which implies $\cH^d(X)=0)$ and $d_{\om_ {{\widetilde\com}_\infty}} (\xi)  \leq s< d$ for every $\xi\in X$. Therefore,  if we set
$$
Z_0:=\{ \xi \in X \cap \d\Omega_{\infty}\backslash (\mathbb{R}^{d}\backslash Q(1)): d_{\om_ {{\widetilde\com}_\infty}} (\xi) \leq  s\}, 
$$
then $\om_ {{\widetilde\com}_\infty}(Z_0)>0$ (recall that by the earlier remark, $ \d\Omega_{\infty}\backslash (\mathbb{R}^{d}\cup Q(1))\subset \d\tilde{\Omega}_{\infty}$). Furthermore,  for $\alpha\in (0,d-s)$, there exists $\rho_0 \leq \min\{\diam \widetilde \com_\infty, 1\}$ such that $\om_{\widetilde\com_\infty}(Z_1)>0$, where
$$
Z_1:=\left\{ \xi \in Z_0: \frac{\log \om_{{\widetilde\com}_\infty} (B(\xi,r)\cap \d \widetilde\com_\infty)}{\log r} < d-\alpha, \,\, \text{for any}\,\, r \in (0, \rho_0]\right\}.
$$
 Notice that this implies that $\om_{\widetilde \com_\infty}(B(\xi,r) \cap \d \widetilde \com_\infty )> r^{d-\alpha}$ for any $\xi \in Z_1$ and $0<r \leq \rho_0 \leq 1$.

 Let us fix $\xi_0 \in Z_1$. By the inner regularity of $\om_{\widetilde\com_\infty}$, there exist $r_0 \in (0, \rho_0]$ and a compact set $E \subset Z_1 \cap B(\xi_0, r_0)$ so that  $\om_{\widetilde \com_\infty}(E)>0$ and  $\om_{\widetilde \com_\infty}(B(\xi,r) \cap \d \widetilde \com_\infty )> r^{d-\alpha}$ for every $\xi \in E$ and $r \in(0,r_0)$. Since $E \subset \d \widetilde \com_\infty \cap \d \com_\infty$ and $\overline{\widetilde \com}_\infty \subset  \overline{\com}_\infty$, in view of the maximum principle, we have that $\om_{\com_\infty}(E)>0$ and $\om_{\com_\infty} ( B(\xi,r) \cap \d \com_\infty ) > r^{d-\alpha}$ for every $\xi \in E$ and $r \in(0,r_0)$. 
 
 We now fix $\varepsilon>0$ small enough. By Lemmas \ref{lemreif1} and \ref{lem:Wolff-Rflat}, if we choose $\theta \in (0,1)$ so that $c_1 \theta\leq\delta_0$, taking into account that $\Omega_\infty$ is $(c_1\theta,\infty)$-Reifenberg flat and hence
 $(\delta_0,\infty)$-Reifenberg flat, we infer that 
there exists a $(c\varepsilon^{1/2}, \infty)$-Reifenberg flat domain $\com$ such that $\overline{\com}_\infty \subset \overline{\com}$ and $E \subset \d\com \cap \d\com_\infty$. Once again, we apply the maximum principle and have that $\om_{\com}(E)>0$ and $\om_{\com} ( B(\xi,r) \cap \d \com_{\infty}) > r^{d-\alpha}$ for every $\xi \in E$ and $r \in(0,r_0)$. Hence, by Theorem \ref{thm:finitelength} we obtain that $\cH^d|_{\d \com}$ is locally finite, that is, Radon. This concludes our theorem since $E$ is a compact subset of $\Omega$ so that $\cH^d(E)=0$ and $\om_\com(E)>0$.
\end{proof}

\end{document}